\documentclass[10pt]{article}

\usepackage{amsmath, amsthm, amssymb, enumerate}
\usepackage[unicode,breaklinks=true,colorlinks=true,linkcolor=blue,urlcolor=blue,citecolor=blue]{hyperref}

\usepackage{marginnote}

\usepackage{tikz}

\usepackage{color}
\usepackage{xcolor}

\usepackage{fancyhdr,lastpage}

\chardef\coloryes=0 
\chardef\isitdraft=0 

\ifnum\isitdraft=1 
\def\eqref#1{({\ref{#1}})}                

\def\startnewsection#1#2{\section{#1}\label{#2}\setcounter{equation}{0}}   
\def\nnewpage{} 
\else

\def\startnewsection#1#2{\section{#1}\label{#2}\setcounter{equation}{0}}   
\def\nnewpage{} 

\fi

\begin{document}
\nocite{*}

\def\ques{{\colr \underline{??????}\colb}}
\def\nto#1{{\colC \footnote{\em \colC #1}}}
\def\fractext#1#2{{#1}/{#2}}
\def\fracsm#1#2{{\textstyle{\frac{#1}{#2}}}}   
\def\nnonumber{}
\def\les{\lesssim}
\def\plusdelta{+\delta}

\def\colr{{}}
\def\colg{{}}
\def\colb{{}}
\def\colu{{}}
\def\cole{{}}
\def\colA{{}}
\def\colB{{}}
\def\colC{{}}
\def\colD{{}}
\def\colE{{}}
\def\colF{{}}


\ifnum\coloryes=1

\definecolor{coloraaaa}{rgb}{0.1,0.2,0.8}
\definecolor{colorbbbb}{rgb}{0.1,0.7,0.1}
\definecolor{colorcccc}{rgb}{0.8,0.3,0.9}
\definecolor{colordddd}{rgb}{0.0,.5,0.0}
\definecolor{coloreeee}{rgb}{0.8,0.3,0.9}
\definecolor{colorffff}{rgb}{0.8,0.3,0.9}
\definecolor{colorgggg}{rgb}{0.5,0.0,0.4}

\def\colb{\color{black}}
\def\colr{\color{red}}
\def\cole{\color{colorgggg}}
\def\colu{\color{blue}}
\def\colg{\color{colordddd}}

\def\colA{\color{coloraaaa}}
\def\colB{\color{colorbbbb}}
\def\colC{\color{colorcccc}}
\def\colD{\color{colordddd}}
\def\colE{\color{coloreeee}}
\def\colF{\color{colorffff}}
\def\colG{\color{colorgggg}}

\fi
\ifnum\isitdraft=1
\chardef\coloryes=1 
\baselineskip=17pt
\pagestyle{myheadings}
\reversemarginpar
\def\const{\mathop{\rm const}\nolimits}  
\def\diam{\mathop{\rm diam}\nolimits}    
\def\rref#1{{\ref{#1}{\rm \tiny \fbox{\tiny #1}}}}
\def\theequation{\fbox{\bf \thesection.\arabic{equation}}}
\def\plusdelta{+\delta}
\def\startnewsection#1#2{\newpage\colg \section{#1}\colb\label{#2}
	\setcounter{equation}{0}
	\pagestyle{fancy}
	\lhead{\colb Section~\ref{#2}, #1 }
	\cfoot{}
	\rfoot{\thepage\ of \pageref{LastPage}}

\chead{}
\rhead{\thepage}
\def\nnewpage{\newpage}
\newcounter{startcurrpage}
\newcounter{currpage}
\def\llll#1{{\rm\tiny\fbox{#1}}}
\def\blackdot{{\color{red}{\hskip-.0truecm\rule[-1mm]{4mm}{4mm}\hskip.2truecm}}\hskip-.3truecm}
\def\bluedot{{\colC {\hskip-.0truecm\rule[-1mm]{4mm}{4mm}\hskip.2truecm}}\hskip-.3truecm}
\def\purpledot{{\colA{\rule[0mm]{4mm}{4mm}}\colb}}
\def\pdot{\purpledot}
\else
\baselineskip=15pt
\def\blackdot{{\rule[-3mm]{8mm}{8mm}}}
\def\purpledot{{\rule[-3mm]{8mm}{8mm}}}
\def\pdot{}
\fi

\def\qq{{\bar q}}
\def\KK{K}
\def\ema#1{\underline{\underline{#1}}}
\def\emb#1{\dotuline{#1}}

\ifnum\isitdraft=1
\def\llabel#1{\nonumber}
\else
\def\llabel#1{\nonumber}
\fi

\def\tepsilon{\tilde\epsilon}
\def\restr{\bigm|}
\def\into{\int_{\Omega}}
\def\intu{\int_{\Gamma_1}}
\def\intl{\int_{\Gamma_0}}
\def\tpar{\tilde\partial}
\def\bpar{\,|\nabla_2|}
\def\barpar{\bar\partial}
\def\FF{F}
\def\gdot{{\color{green}{\hskip-.0truecm\rule[-1mm]{4mm}{4mm}\hskip.2truecm}}\hskip-.3truecm}
\def\bdot{{\color{blue}{\hskip-.0truecm\rule[-1mm]{4mm}{4mm}\hskip.2truecm}}\hskip-.3truecm}
\def\cydot{{\color{cyan} {\hskip-.0truecm\rule[-1mm]{4mm}{4mm}\hskip.2truecm}}\hskip-.3truecm}
\def\rdot{{\color{red} {\hskip-.0truecm\rule[-1mm]{4mm}{4mm}\hskip.2truecm}}\hskip-.3truecm}

\def\nts#1{{\color{red}\hbox{\bf ~#1~}}} 

\def\ntsr#1{\vskip.0truecm{\color{red}\hbox{\bf ~#1~}}\vskip0truecm} 

\def\ntsf#1{\footnote{\hbox{\bf ~#1~}}} 
\def\ntsf#1{\footnote{\color{blue}\hbox{\bf ~#1~}}} 
\def\bigline#1{~\\\hskip2truecm~~~~{#1}{#1}{#1}{#1}{#1}{#1}{#1}{#1}{#1}{#1}{#1}{#1}{#1}{#1}{#1}{#1}{#1}{#1}{#1}{#1}{#1}\\}
\def\biglineb{\bigline{$\downarrow\,$ $\downarrow\,$}}
\def\biglinem{\bigline{---}}
\def\biglinee{\bigline{$\uparrow\,$ $\uparrow\,$}}
\def\ceil#1{\lceil #1 \rceil}
\def\gdot{{\color{green}{\hskip-.0truecm\rule[-1mm]{4mm}{4mm}\hskip.2truecm}}\hskip-.3truecm}
\def\bluedot{{\color{blue} {\hskip-.0truecm\rule[-1mm]{4mm}{4mm}\hskip.2truecm}}\hskip-.3truecm}
\def\rdot{{\color{red} {\hskip-.0truecm\rule[-1mm]{4mm}{4mm}\hskip.2truecm}}\hskip-.3truecm}
\def\dbar{\bar{\partial}}
\newtheorem{Theorem}{Theorem}[section]
\newtheorem{Corollary}[Theorem]{Corollary}
\newtheorem{Proposition}[Theorem]{Proposition}
\newtheorem{Lemma}[Theorem]{Lemma}
\theoremstyle{remark}
\newtheorem{Remark}[Theorem]{Remark}
\newtheorem{definition}{Definition}[section]
\def\theequation{\thesection.\arabic{equation}}

\def\sqrtg{\sqrt{g}}
\def\dd{\delta}
\def\EE{{\mathcal E}}
\def\lot{{\rm l.o.t.}}                       
\def\inon#1{~~~\hbox{#1}}                
\def\endproof{\hfill$\Box$\\}
\def\square{\hfill$\Box$\\}
\def\inon#1{~~~\hbox{#1}}                
\def\comma{ {\rm ,\qquad{}} }            
\def\commaone{ {\rm ,\qquad{}} }         
\def\dist{\mathop{\rm dist}\nolimits}    
\def\sgn{\mathop{\rm sgn\,}\nolimits}    
\def\Tr{\mathop{\rm Tr}\nolimits}    
\def\dive{\mathop{\rm div}\nolimits}    
\def\grad{\mathop{\rm grad}\nolimits}    
\def\curl{\mathop{\rm curl}\nolimits}    
\def\det{\mathop{\rm det}\nolimits}    
\def\supp{\mathop{\rm supp}\nolimits}    
\def\re{\mathop{\rm {\mathbb R}e}\nolimits}    
\def\wb{\bar{\omega}}
\def\Wb{\bar{W}}
\def\indeq{\qquad{}}                     
\def\indeqtimes{\indeq\indeq\times} 
\def\period{.}                           
\def\semicolon{\,;}                      
\newcommand{\cD}{\mathcal{D}}
\newcommand{\cH}{\mathcal{H}}
\newcommand{\al}{\alpha}
\newcommand{\be}{\beta}
\newcommand{\ga}{\gamma}
\newcommand{\de}{\delta}
\newcommand{\ep}{\epsilon}
\newcommand{\si}{\sigma}
\newcommand{\Si}{\Sigma}
\newcommand{\vfi}{\varphi}
\newcommand{\om}{\omega}
\newcommand{\Om}{\Omega}
\newcommand{\cqd}{\hfill $\qed$\\ \medskip}
\newcommand{\rar}{\rightarrow}
\newcommand{\imp}{\Rightarrow}
\newcommand{\tr}{\operatorname{tr}}
\newcommand{\vol}{\operatorname{vol}}
\newcommand{\id}{\operatorname{id}}
\newcommand{\p}{\parallel}
\newcommand{\norm}[1]{\Vert#1\Vert}
\newcommand{\abs}[1]{\vert#1\vert}
\newcommand{\nnorm}[1]{\left\Vert#1\right\Vert}
\newcommand{\aabs}[1]{\left\vert#1\right\vert}
\newcommand{\footremember}[2]{%
	\footnote{#2}
	\newcounter{#1}
	\setcounter{#1}{\value{footnote}}%
}
\newcommand{\footrecall}[1]{%
	\footnotemark[\value{#1}]%
}

	\title{Norm inflation for the Boussinesq system}
	\author{%
		Zongyuan Li\footremember{Division of Applied Mathematics}{Division of Applied Mathematics, Brown University, email: zongyuan\_li@brown.edu.}%
		\and Weinan Wang\footremember{Department of Mathematics}{Department of Mathematics, University of Southern California, email: wangwein@usc.edu.}%
	}
	\date{}
	\maketitle
	
	\bigskip

	\bigskip

	\begin{abstract}
		We prove the norm inflation phenomena for the Boussinesq system on $\mathbb T^3$. For arbitrarily small initial data $(u_0,\rho_0)$ in the negative-order Besov spaces $\dot{B}^{-1}_{\infty, \infty} \times \dot{B}^{-1}_{\infty, \infty}$, the solution can become arbitrarily large in a short time. Such largeness can be detected in $\rho$ in Besov spaces of any negative order: $\dot{B}^{-s}_{\infty, \infty}$ for any $s>0$. Notice that our initial data space is scaling critical for $u$ and is scaling subcritical for $\rho$.
	\end{abstract}

	

	\section{Introduction}
	In this paper, we consider the Boussinesq system on $\mathbb T^3$
	\begin{align}
	& u_t 
	-\Delta u 
	+ u\cdot \nabla u
	+ \nabla \pi
	=
	\rho {e}_{3}
	\label{EQ01}
	\\ 
	& \rho_t - \Delta \rho+ u\cdot \nabla \rho = 0
	\label{EQ02}
	\\
	& \nabla \cdot u=0
	.
	\label{EQ03}
	\end{align} 
	Here, $u$ is the velocity satisfying the 3D Navier-Stokes equations
	\cite{DG, T} 
	driven by $\rho$, which represents the
	density or temperature of the fluid, depending on the physical
	context. Also $\pi$ denotes the pressure while $e_3 = (0, 0, 1)$ is the unit vector.
	
	In recent years, there has been extensive research on the $2D$ Boussinesq equations. Researchers have been studying the global existence of solutions and persistence of regularity since the seminal work of Chae \cite{C} and of Hou and Li \cite{HL}, who proved the global existence and uniqueness of solutions. In \cite{CW, SW}, the authors addressed the global well-posedness in Sobolev spaces $H^s$.  Kukavica and the second author of this paper addressed the global Sobolev persistence of regularity in $W^{s,q} \times W^{s,q}$ for the fractional Boussinesq system in \cite{KW1} and the long time behavior of solutions in \cite{KW2}, respectively. For other global well-posedness results on the Boussinesq equations, cf. \cite{HK1, W1, W2, WY2}.
	
The $3D$ Boussinesq system possesses a natural scaling, i.e., for $\lambda>0$, 
	 \begin{equation}
	 u_{\lambda}(x, t)
	 =
	 \lambda u(\lambda x, \lambda^{2}t),\quad \rho_{\lambda}(x, t)
	 =
	 \lambda^{3} \rho(\lambda x, \lambda^{2}t)
	 \llabel{EQ19}
	 \end{equation}
     are again solutions with initial data
     \begin{equation}
     u_{\lambda}(x, 0)
     =
     \lambda u(\lambda x),\quad \rho_{\lambda}(x, 0)
     =
     \lambda^{3} \rho(\lambda x)
     .
     \llabel{EQ19}
     \end{equation}
Note that $L_3\times L_1$ is scaling invariant space for the initial data:
\begin{equation*}
\|u_{\lambda}(\cdot, 0)\|_{L^3} = \|\lambda u_0(\lambda\cdot)\|_{L^3}=\|u_0\|_{L^3},\quad \|\rho_{\lambda}(\cdot, 0)\|_{L^1} = \|\lambda^3 \rho_0(\lambda\cdot)\|_{L^3}=\|\rho_0\|_{L^1}.
\end{equation*}
Much effort has been made in search for the largest initial data space for the local and global well-posedness. For the Navier-Stokes equations, Koch and Tataru in \cite{KT} obtained the local well-posedness result with small $BMO^{-1}$ initial data. Such space is the largest scaling critical space for the local well-posedness of the NSE. In fact, \cite{BP} and \cite{Y} showed the local ill-posedness for initial spaces close to $BMO^{-1}$: $B^{-1}_{\infty,q}$, $q\in(2,\infty]$. It is worth mentioning that, for supercritical initial data spaces, we can still discuss some well-posedness results, only in the sense of ``almost surely''. See for example, \cite{NPS, WY1}.

Back to our Boussinesq system \eqref{EQ01}-\eqref{EQ03}. Let us first mention some results in the positive direction. As mentioned in \cite{BH}, by adapting Kato's method, the local well-posedness result can be proved for scaling invariant $L_3\times L_1$ initial data, given that the initial data is sufficiently small. The mild formulation for the Boussinesq system is given in \eqref{eqn-1204-2139}. One can also find some discussions of the uniqueness and long time behaviors in \cite{BH} and \cite{BM}. 

In this paper, we show a local ill-posedness result for the $3D$ Boussinesq system. Our initial data space is $u_0\in\dot{B}^{-1}_{\infty, \infty}$ which is the largest possible scaling invariant space for the mild formulation, and $\rho_0\in\dot{B}^{-1}_{\infty, \infty}$ which is scaling subcritical.Specifically, we show the existence of a smooth space-periodic solution with arbitrarily small initial data $(u_{0},\rho_{0})\in \dot{B}^{-1}_{\infty, \infty} \times \dot{B}^{-1}_{\infty, \infty}$ that becomes arbitrarily large in $\dot{B}^{-s}_{\infty, \infty}$ for all $s > 0$ immediately. This type of result is called the ``norm inflation''. This phenomenon was discovered by Bourgain and Pavlovi\'c in \cite{BP}, concerning the local ill-posedness for the 3D Navier-Stokes equations. The main idea of their approach is to construct initial data $u_{0}$ in $\dot{B}^{-1}_{\infty, \infty}$ so that when they evolve in time, one particular frequency will accumulate while others will dissipate. 

There have been many other analogous norm inflation results on different fluid models. In \cite{DQS}, Dai, Qing, and Schonbek proved norm inflation for the MHD equations. In \cite{CD1, CD2}, Cheskidov and Dai proved the norm inflation occurrence for the generalized Navier-Stokes and the generalized MHD equations. In this paper, we prove the ill-posedness result for the Boussinesq system by constructing initial data which leads to the norm inflation. To the best knowledge of the authors, this is the first result addressing the norm inflation for the Boussinesq equations. Note that the ill posedness results obtained in \cite{CL, GHK} are for a dispersive type system which is different from the equations considered in this paper.
The main result of this paper is the following.
	\begin{Theorem}
		\label{thm-1204-1740}
		For sufficiently small $\epsilon>0$, there exists a solution $(u(t),\rho(t))$ to the system \eqref{EQ01}--\eqref{EQ03} such that
		\begin{equation}
		\Vert u(0) \Vert_{\dot{B}^{-1}_{\infty, \infty}}
		+
		\Vert \rho(0) \Vert_{\dot{B}^{-1}_{\infty, \infty}}
		<
		\epsilon
		\llabel{EQ19}
		\end{equation}
		and
		\begin{equation}\label{est-1204-1822}
		\Vert \rho(T) \Vert_{\dot{B}^{-s}_{\infty, \infty}}
		>
		1/\epsilon,
		\end{equation}
		where $0<T<\epsilon$ and $s>0$.
	\end{Theorem}
The organization of this paper goes as follows. In Section \ref{sec-pre}, we introduce our notation and the mild formulation. The rest of the paper is devoted to the proof of Theorem \ref{thm-1204-1740}. In Section \ref{sec-interaction}, we first construct the initial data as a combination of plane waves with very lacunary frequencies and then compute their interactions. The estimates are given in Section \ref{sec-1207-1609} and \ref{sec-1207-1609-2}.  In Section \ref{sec-prof-main}, we close the estimates by giving the proof of Theorem \ref{thm-1204-1740}.
	
	\section{Preliminaries and mild formulation}\label{sec-pre}
Recall the Leray projector
	\begin{equation}
	\mathbb P
	=
	I + \nabla (-\Delta)^{-1} \nabla \cdot
	\llabel{EQ071}
	\end{equation} 
as well as its symbol
	\begin{equation}
	\widehat{(\mathbb Pu)}_{j}(\xi)
	=
	\left(\delta_{jk} - \frac{\xi_j \xi_k}{|\xi|^2}\right)\hat{u}_{k}(\xi)
	\comma
	j=1,2,3.
	\llabel{EQ08}
	\end{equation} 
We start from the mild formulation of $(u,\rho)$ and rewrite the Boussinesq system \eqref{EQ01}--\eqref{EQ03} as in \cite{BH}:
     \begin{equation}\label{eqn-1204-2139}
     \begin{cases}
     u = e^{t\Delta}(u_0 + t\mathbb{P}\rho_0\vec{e}_3) + B_1(u,u) + B_2(u,\rho),\\
     \rho = e^{t\Delta}\rho_0 + B_3(u,\rho).
     \end{cases}
     \end{equation}
Here the bilinear forms are given by
\begin{equation}\label{eqn-1204-1601}
\begin{split}
	B_1(u,v)
	&=
	-\int_{0}^{t} e^{(t-s)\Delta} \mathbb P \nabla \cdot (u \otimes v)\, ds,\\
	B_2(u,\theta)
	&=
	-\int_{0}^{t} e^{(t-s)\Delta} (t-s) \mathbb P (\nabla \cdot (u \theta))\vec{e}_3\, ds,\\
B_3(u,v)
	&=
	-\int_{0}^{t} e^{(t-s)\Delta} \nabla \cdot (u \theta)\, ds.
\end{split}
\end{equation}
Such mild formation is equivalent to the original system. It can be obtained as follows. We first rewrite \eqref{EQ01} and \eqref{EQ02} in the integral form, by applying the Leray projection and the Duhamel principle. Then we replace the linear term in the equation of $u$ by the integral form equation of $\rho$ to make all the integral terms to be bilinear. For details, cf. \cite{BH}.
	Next, we define the Besov norms for $s>0$ as in \cite{L}:
	\begin{equation}
	\Vert f \Vert_{\dot{B}^{-s}_{\infty, \infty}}
	=
	\sup_{t>0} t^{\frac{s}{2}}\Vert e^{t\Delta}f \Vert_{L^{ \infty}}
	\llabel{EQ071}
	\end{equation} 
	\begin{equation}
	\Vert f \Vert_{{B}^{-s}_{\infty, \infty}}
	=
	\sup_{0<t<1} t^{\frac{s}{2}}\Vert e^{t\Delta}f \Vert_{L^{ \infty}}
	.
	\llabel{EQ071}
	\end{equation} 
	Note that they are equivalent on torus. Now we state the estimates for the bilinear operator. The following lemma plays a key role in estimates with Leray projectors.
		\cole
\begin{Lemma}\label{lem-1209-1242}
	There exists a constant $c$ such that
	\begin{equation*}
	\Vert \nabla e^{t\Delta}\mathbb Pf \Vert_{L^{\infty}}
	\leq
	Ct^{-1/2}
	\Vert f \Vert_{L^{\infty}}
	,
	\end{equation*}
	for $t>0$ and $f \in L^{\infty}$.
\end{Lemma}
\cole
As shown in \cite{L}, the following bilinear estimate follows.
\begin{Lemma}\cite{L}\label{lem-1204-1626}
	For any $u, v \in L^{1}(0,T; L^{\infty})$, $B_1$ satisfies
	\begin{equation*}
	\Vert B_1(u,v) \Vert_{L^{\infty}}
	\leq
	C
	\int_{0}^{t} \frac{1}{(t-s)^{1/2}} \Vert u \Vert_{L^{\infty}}
	\Vert v \Vert_{L^{\infty}}\, ds
	.
	\llabel{EQ071}
	\end{equation*}
\end{Lemma}
In the same spirit, we can also bound $B_2$ and $B_3$ as
\begin{equation}\label{eqn-1208-2243}
\Vert B_2(u,\theta) \Vert_{L^{\infty}}
\leq
C
\int_{0}^{t} (t-s)^{1/2} \Vert u \Vert_{L^{\infty}}
\Vert \theta \Vert_{L^{\infty}}\, ds
\end{equation}
and
\begin{equation}\label{eqn-1209-0022}
\Vert B_3(u,\theta) \Vert_{L^{\infty}}
\leq
C
\int_{0}^{t} \frac{1}{(t-s)^{1/2}} \Vert u \Vert_{L^{\infty}}
\Vert \theta \Vert_{L^{\infty}}\, ds
.
\end{equation}
The proof for \eqref{eqn-1208-2243} is the same with Lemma \ref{lem-1204-1626}, by using Lemma~\ref{lem-1209-1242}. The proof for \eqref{eqn-1209-0022} is more straightforward since there is no Leray projector in $B_3$. For details, refer to \cite{L}.

Throughout this paper, we use notation $\lesssim$, $\gtrsim$, and $\approx$ when we have the inequality in the corresponding direction with constants independent of the parameters that we are interested in, e.g., $r$ and $t$.

	\section{Interaction of plane wave}\label{sec-interaction}
\subsection{The first approximation in Picard's iteration}
Similarly to \cite{BP}, we decompose the solution as
     \begin{equation}\label{eqn-1204-1555}
\begin{split}
     u
     &=	 g + u_1 + y,\\
     \rho
     &= \theta + \rho_1 + z,
\end{split}
\end{equation}
      where
      \begin{equation}
	  \begin{split}
g&:= e^{t\Delta}(u_0 + t\mathbb{P}\rho_0\vec{e}_3),\\
\theta&:= e^{t\Delta}\rho_0,\\
      u_1 &:= B_1(g, g) + B_2(g, \theta),\\
	  \rho_1 &:= B_3(g, \theta).\label{eqn-1202-0008}\\
	 \end{split}
	   \end{equation}
To facilitate our computation, we further decompose
\begin{equation}
y = y_0 + y_1 + y_2,\quad z = z_0 + z_1 + z_2,
\llabel{EQ19}
\end{equation}
where
\begin{equation}\label{eqn-1204-1626}
\begin{split}
y_0 &= 2B_1(g,u_1) + B_1(u_1,u_1) + B_2(g,\rho_1) + B_2(u_1,\theta) + B_2(u_1,\rho_1),\\
y_1 &= 2B_1(g,y) + 2B_1(u_1,y) + B_2(g,z) + B_2(y,\theta) + B_2(y,\rho_1) + B_2(u_1,z),\\
y_2 &= B_1(y,y) + B_2(y,z),\\
z_0 &= B_3(g,\rho_1) + B_3(u_1,\theta) + B_3(u_1,\rho_1),\\
z_1 &= B_3(g,z) + B_3(y,\theta) + B_3(y,\rho_1) + B_3(u_1,z),\\
z_2 &= B_3(y,z).
\end{split}
\end{equation}
	     \subsection{Choice of initial data and its diffusion}
	     We construct the initial data $(u_{0}, \rho_{0})$ as
	     \begin{equation}
	     u_{0}
	     =
	     r^{-\beta} \sum_{s=1}^{r}
	     |k_{s}| v_{s}\cos{(k_{s}\cdot x)},
	     \llabel{EQ19}
	     \end{equation}
	     \begin{equation}
	     \rho_{0}
	     =
	     r^{-\beta} \sum_{s=1}^{r}
	     |k'_{s}|\cos{(k'_{s}\cdot x)}
	     ,
	     \llabel{EQ19}
	     \end{equation}
where $r$ is a large enough number to be chosen at the end and $\beta >0$. The frequencies $k_i$ and $k_i'$ are chosen as follows. We first fix a large number $K>0$, and then for $i \in \mathbb N^{+}$ we define
\begin{equation}\label{eqn-1204-1824}
\overline{k_i'}=2^{i-1}K,\,\,k_i'=(0,0,\overline{k_i'}),\quad k_i=k_i'+\eta,
\end{equation}
where $\eta=(0,1,0)$. The vector $v_i$ is given as $$v_i=(0,\frac{1}{2},-1/(2\overline{k_i'})).$$
Clearly 
\begin{equation}\label{eqn-1206-2231}
v_i\cdot k_i=0,\,\, v_i\cdot k_i'=-1/2,\,\,\text{and}\,\, v_i\cdot k_j'=-\frac{|k_j'|}{2|k_i'|}\approx v_i\cdot k_j.
\end{equation}
From our construction, we have
\begin{equation*}
\dive(u_0)=0,\quad\text{and}\,\,\mathbb{P}(\rho_0 \vec{e_3})=0.
\end{equation*}
The following lemma contains useful properties which are used in our computation below. The proof is by direct computation, or one may check \cite{CD1}.
\begin{Lemma}\cite{CD1}
	\label{L04}
	For $\gamma >0$ and $k_{i}$ constructed as above, we have the following relations:
	\begin{equation}
	\sum_{j<i} |k_{j}|^\gamma
	\sim
	|k_i|^\gamma
	\llabel{EQ19}
	\end{equation}
	and
	\begin{equation}
	\sum_{i=1}^{r} |k_{i}|^{\gamma} e^{-|k_{i}|^{2}t}
	\les
	t^{\gamma/2}
	.
	\llabel{EQ19}
	\end{equation}
\end{Lemma}
 Note that the diffusion of a plane wave $v\cos{(k\cdot x)}$ under the Laplacian $\Delta$ is given by
	     \begin{equation}
	     e^{t\Delta} v\cos{(k\cdot x)}
	     =
	     e^{-|k|^{2}t} v\cos{(k\cdot x)}
	     .
	     \llabel{EQ19}
	     \end{equation}
	     Therefore, we may write the diffusion of the initial data
	     \begin{align}
	     e^{t\Delta}u_{0}
	     =
	     r^{-\beta} \sum_{s=1}^{r}
	     |k_{s}| e^{-|k|^{2}t} v_{s}\cos{(k_{s}\cdot x)}
	     ,
	     \label{eqn-1204-1610-1}
	     \end{align}
	     \begin{align}
	     e^{t\Delta}\rho_{0}
	     =
	     r^{-\beta} \sum_{s=1}^{r}
	     |k'_{s}|e^{-|k|^{2}t} \cos{(k'_{s}\cdot x)}
	     .
	     \label{eqn-1204-1611-2}
	     \end{align}
To complete this section, we provide some estimates for the initial data and their diffusion.
\begin{Lemma}\label{lem-1204-1830}
For $\beta>0$, the inequalities
\begin{align}
&\|u_0\|_{\dot{B}^{-1}_{\infty, \infty}} \approx r^{-\beta},\quad \|\rho_0\|_{\dot{B}^{-1}_{\infty, \infty}} \approx r^{-\beta},\label{est-1203-1959}\\
&\|e^{t\Delta}u_0\|_{L^\infty}\les r^{-\beta}t^{-1/2},\quad\|e^{t\Delta}\rho_0\|_{L^\infty}\les r^{-\beta}t^{-1/2},\label{est-1204-1546}
\end{align}
hold for all $t\in (0,1]$.
\end{Lemma}
\begin{proof}
The estimate \eqref{est-1203-1959} is clear from our construction. In the following, we prove estimate \eqref{est-1204-1546} for the diffusion. By Lemma \ref{L04},
\begin{equation*}
\begin{split}
|e^{t\Delta}u_0| &= \big|r^{-\beta}\sum_{s=1}^r e^{-t|k_s|^2}|k_s| v_s \cos{(k_s\cdot x)}\big|\\
&\lesssim r^{-\beta}t^{-1/2} \sum_{s=1}^r e^{-t|k_s|^2}(|k_s|t^{1/2}) \lesssim r^{-\beta}t^{-1/2},
\end{split}
\end{equation*}
and similarly, 
\begin{equation*}
\begin{split}
|e^{t\Delta}\rho_0| &= |r^{-\beta}\sum_{s=1}^r e^{-t|k_s|^2}|k_s| v_s \cos{(k_s\cdot x)}|\\
&\lesssim r^{-\beta}t^{-1/2} \sum_{s=1}^r e^{-t|k_s|^2}|k_s|t^{1/2}\lesssim r^{-\beta}t^{-1/2}.
\end{split}
\end{equation*}
\end{proof}

\subsection{Estimate for $u_1$ and $\rho_1$}\label{sec-1207-1609}
In this section, we estimate the most important terms in \eqref{eqn-1204-1555}, $u_1$ and $\rho_1$ by computing the three bilinear terms directly. Recall the definitions of $u_1$ and $\rho_1$ in \eqref{eqn-1202-0008} and the bilinear forms in \eqref{eqn-1204-1601}.

\textbf{Estimate of $B_{1}(e^{t\Delta}u_{0}, e^{t\Delta}u_{0})$}. This term only involves the interactions between $\cos{(k_i\cdot x)}$ and $\cos{(k_j\cdot x)}$. From \eqref{eqn-1204-1610-1}, we further compute
	     \begin{align*}
	     e^{t\Delta}u_{0}\cdot\nabla e^{t\Delta}u_{0}
	     =-
	     r^{-2\beta} \sum_{i,j=1}^{r}
	     |k_{i}||k_{j}| (v_{i}\cdot k_{j})e^{-t(|k_i|^2+|k_j|^2)}v_{j}\cos{(k_{i}\cdot x)}\sin{(k_{j}\cdot x)}
	     &\\=
	     -\frac{r^{-2\beta}}{2} \sum_{i,j=1}^{r}
	     |k_{i}||k_{j}| (v_{i}\cdot k_{j})e^{-t(|k_i|^2+|k_j|^2)}v_{j}\sin{((k_{i}+k_{j})\cdot x)}
	     &\\-
	     \frac{r^{-2\beta}}{2} \sum_{i,j=1}^{r}
	     |k_{i}||k_{j}| (v_{i}\cdot k_{j})e^{-t(|k_i|^2+|k_j|^2)}v_{j}\sin{((k_{i}-k_{j})\cdot x)}
	     ,
	     \end{align*}
	     where we used the fact: $\cos{a}\sin{b}=\left((\sin(a+b)-\sin(a-b)) \right)/2$.
	     Furthermore, we have 
	     \begin{align*}
	     \mathbb P e^{t\Delta}u_{0}\cdot\nabla e^{t\Delta}u_{0}
	     =
	    &-\frac{r^{-2\beta}}{2} \sum_{i,j=1}^{r}
	    |k_{i}||k_{j}| (v_{i}\cdot k_{j})e^{-t(|k_i|^2+|k_j|^2)}w_{i,j}\sin{((k_{i}+k_{j})\cdot x)}
	    \\&-
	    \frac{r^{-2\beta}}{2} \sum_{i,j=1}^{r}
	    |k_{i}||k_{j}| (v_{i}\cdot k_{j})e^{-t(|k_i|^2+|k_j|^2)}\tilde{w}_{i,j}\sin{((k_{i}-k_{j})\cdot x)}
	    \\=&
	     E_1 + E_2
	     ,
	     \end{align*}
	     where $w_{i,j}$ is the projection of $v_{j}$ onto the orthogonal to $k_{i} + k_{j}$ and $\tilde{w}_{i,j}$ is the projection of $v_{j}$ onto the orthogonal to $k_{i} - k_{j}$.
	     Therefore, 
	     \begin{align}
	     B_{1}(e^{t\Delta}u_{0}, e^{t\Delta}u_{0})
	     =
	     \int_{0}^{t} e^{(t-s)\Delta}E_{1}\,ds
	     +
	     \int_{0}^{t} e^{(t-s)\Delta}E_{2}\,ds
	     =
	     F_1 + F_2
	     .
	     \llabel{EQ19}
	     \end{align}
	     For $F_1$, we have
	     \begin{align}
	     F_1
	     =
	     \frac{r^{-2\beta}}{2} \sum_{i,j=1}^{r}
	     |k_{i}||k_{j}| (v_{i}\cdot k_{j})e^{-t(|k_i|^2+|k_j|^2)}w_{i,j}\sin{((k_{i}+k_{j})\cdot x)}
	     \frac{1 - e^{-t(|k_i + k_j|^2-|k_i|^2-|k_j|^2)}}{|k_i + k_j|^2 - (|k_i|^2+|k_j|^2)}
	     .
	     \llabel{EQ19}
	     \end{align}
	     By the facts that $(1-e^{-x})/x$ and $xe^{-x}$ are bounded for $x>0$ and $|v_i\cdot k_j|\approx |k_j|/|k_i|$, we obtain
	     \begin{equation}\label{est-1206-2244}
\begin{split}
	     |F_1|
	     &\les
	     \frac{r^{-2\beta}}{2} \sum_{j=1}^{r}\sum_{i\leq j}^{}
	     te^{-t|k_i|^2}|k_{j}|^{2}e^{-t|k_j|^2} + \frac{r^{-2\beta}}{2} \sum_{i=1}^{r}\sum_{j<i}^{}
	     te^{-t|k_i|^2}|k_{j}|^{2}e^{-t|k_j|^2}\\
	     &\les
	     \frac{r^{-2\beta}}{2} \sum_{j=1}^{r}j
	     |k_j|^{2}e^{-t\frac{|k_j|^{2}}{4}} + \sum_{i=1}^{r}
	     te^{-t|k_i|^2}|k_{i}|^{2}\\
&\lesssim r^{-2\beta}|\log{t}|\lesssim r^{-2\beta}t^{-\delta}
	     .
\end{split}
	     \llabel{EQ19}
	     \end{equation}
From here on, we fix the constant $\delta\ll 1$. The estimates of $F_2$ follow similarly and we omit the details here. Hence, we have
\begin{equation}\label{est-1204-2222}
\|B_{1}(e^{t\Delta}u_{0}, e^{t\Delta}u_{0})\|_{L^\infty}\lesssim r^{-2\beta}r^{-\delta}.
\end{equation}

\textbf{The estimate of $B_{2}(e^{t\Delta}u_{0}, e^{t\Delta}\rho_{0})$ and $B_{3}(e^{t\Delta}u_{0}, e^{t\Delta}\rho_{0})$}. Both terms concern the interactions between $\cos{(k_i\cdot x)}$ and $\cos{(k_j'\cdot x)}$. Here we first compute for $B_3$, since it contains the most important term in which the accumulation occurs.

From \eqref{eqn-1204-1611-2}, we deduce
	     \begin{align*}
	     e^{t\Delta}u_{0}\cdot\nabla e^{t\Delta}\rho_{0}
	     =
	     &-
	     r^{-\beta-\beta} \sum_{i,j=1}^{r}
	     |k_{i}||k'_{j}| (v_{i}\cdot k'_{j})e^{-t(|k_i|^2+|k'_j|^2)}\cos{(k_{i}\cdot x)}\sin{(k'_{j}\cdot x)}
	     \\=& -
	     \frac{r^{-2\beta} }{2} \sum_{i,j=1}^{r}
	     |k_{i}||k'_{j}| (v_{i}\cdot k'_{j})e^{-t(|k_i|^2+|k'_j|^2)}\sin{((k_{i}-k'_{j})\cdot x)}
	     \\
	     &-\frac{r^{-2\beta} }{2} \sum_{i,j=1}^{r}
	     |k_{i}||k'_{j}| (v_{i}\cdot k'_{j})e^{-t(|k_i|^2+|k'_j|^2)}\sin{((k_{i}+k'_{j})\cdot x)}\\
=&-\frac{r^{-2\beta} }{2}\sum_{i=1}^{r}
	     |k_{i}||k'_{i}| (v_{i}\cdot k'_{i})e^{-t(|k_i|^2+|k'_i|^2)}\sin{(\eta\cdot x)}\\
&-\frac{r^{-2\beta} }{2}\sum_{i,j=1, i\neq j}^{r}
	     |k_{i}||k'_{j}| (v_{i}\cdot k'_{j})e^{-t(|k_i|^2+|k'_j|^2)}\sin{((k_{i}-k'_{j})\cdot x)}\\
&-\frac{r^{-2\beta} }{2}\sum_{i,j=1}^{r}
	     |k_{i}||k'_{j}| (v_{i}\cdot k'_{j})e^{-t(|k_i|^2+|k'_j|^2)}\sin{((k_{i}+k'_{j})\cdot x)}.
	     \end{align*}
Now, substituting back to \eqref{eqn-1202-0008}, computing the integration, and noting \eqref{eqn-1206-2231}, we rewrite $\rho_1$ as
\begin{equation*}
\begin{split}
\rho_1=\rho_{1,0} + \rho_{1,1} + \rho_{1,2},
\end{split}
\end{equation*}
where
\begin{equation*}
\begin{split}
\rho_{1,0}:=&\frac{r^{-2\beta} }{4}\sum_{i=1}^{r} e^{-t|\eta|^2}
	     |k_{i}|^2 \frac{1-e^{-t(|k_i|^2+|k'_i|^2)}}{|k_i|^2+|k'_i|^2-1}\sin{(\eta\cdot x)}\\
\rho_{1,1}:=-&\frac{r^{-2\beta}}{2}\sum_{i=1,i\neq j}^{r} e^{-t|k_{i}-k'_{j}|^2}
	     |k_{i}||k'_{j}|(v_{i}\cdot k'_{j}) \frac{1-e^{-t(|k_i|^2+|k'_j|^2)}}{|k_i|^2+|k'_j|^2-1}\sin{((k_{i}-k'_{j})\cdot x)}\\
\rho_{1,2}:=&-\frac{r^{-2\beta}}{2}\sum_{i,j=1}^{r} e^{-t|k_{i}+k'_{j}|^2}
	     |k_{i}||k'_{j}|(v_{i}\cdot k'_{j}) \frac{1-e^{-t(|k_i|^2+|k'_j|^2)}}{|k_i|^2+|k'_j|^2-1}\sin{((k_{i}+k'_{j})\cdot x)}.\\
\end{split}
\end{equation*}
We next bound $\rho_{1,0}$ from both below and above.
\begin{Lemma}\label{lem-1202-2023-1}
\begin{equation*}
\begin{split}
&\|\rho_{1,0}(\cdot,t)\|_{\dot{B}^{-s}_{\infty,\infty}}\gtrsim r^{1-2\beta} ,\quad \forall t\in [K^{-2},1],\\
&\|\rho_{1,0}(\cdot,t)\|_{L^\infty}\les r^{1-2\beta} ,\quad \forall t\in (0,1].
\end{split}
\end{equation*}
\end{Lemma}
\begin{proof}
For $t\geq K^{-2}$, we have $-t(|k_i|^2 + |k_i'|^2)\lesssim -2^{i-1}$. Noting that $|\eta|=1$, we have
\begin{equation*}
\begin{split}
\|\rho_{1,0}(\cdot,t)\|_{\dot{B}^{-s}_{\infty,\infty}} &= \frac{r^{-2\beta} }{4}\Big|\sum_{i=1}^{r} e^{-t|\eta|^2}
	     |k_{i}|^2 \frac{1-e^{-t(|k_i|^2+|k'_i|^2)}}{|k_i|^2+|k'_i|^2-1}\Big|\|\sin{(\eta\cdot x)}\|_{\dot{B}^{-s}_{\infty,\infty}}\\
&\gtrsim r^{-2\beta} \sum_{i=1}^{r} e^{-t} (1-e^{-2^{i-1}}) \gtrsim r^{1-2\beta} .
\end{split}
\end{equation*}
Here we used
\begin{equation*}
|k_i|^2\approx |k_i|^2 + |k_i'|^2 -1\,\,\text{and}\,\,e^{-t}\approx 1,\,\,\text{for}\,\,t\in[0,1].
\end{equation*}
Similarly, if we simply bound $1-e^{-t(|k_i|^2+|k'_i|^2)}$ by $1$, then
\begin{equation*}
\|\rho_{1,0}(\cdot,t)\|_{L^\infty} \lesssim r^{-2\beta}  (\sum_{i=1}^r 1) \lesssim r^{1-2\beta} .
\end{equation*}
\end{proof}
For $\rho_{1,1}$ and $\rho_{1,2}$, we only estimate the upper bound.
\begin{Lemma}\label{lem-1202-2023-2}
\begin{equation*}
\|\rho_{1,1}(\cdot,t)\|_{L^\infty}\les r^{-2\beta}t^{-\delta} ,\quad\|\rho_{1,2}(\cdot,t)\|_{L^\infty}\les r^{-2\beta}t^{-\delta} ,\quad t\in (0,1].
\end{equation*}
\end{Lemma}
\begin{proof}
Here we only prove the inequality for $\rho_{1,1}$, since the computation for $\rho_{1,2}$ is similar. First,
\begin{equation*}
\begin{split}
|\rho_{1,1}| &\lesssim r^{-2\beta}  \sum_{i=1}^r \sum_{j<i} |k_i| |k_j'|(v_{i}\cdot k'_{j})e^{-t|k_i-k_j'|^2}\frac{1 - e^{-(|k_i|^2+|k_j'|^2 -|k_i-k_j'|^2)t}}{|k_i|^2 + |k_j'|^2 - |k_i - k_j'|^2} |\sin{(k_i-k_j')\cdot x}|\\
&\indent+ r^{-2\beta}  \sum_{j=1}^r \sum_{i<j} |k_i| |k_j'|(v_{i}\cdot k'_{j})e^{-t|k_i-k_j'|^2}\frac{1 - e^{-(|k_i|^2+|k_j'|^2 -|k_i-k_j'|^2)t}}{|k_i|^2 + |k_j'|^2 - |k_i - k_j'|^2} |\sin{(k_i-k_j')\cdot x}|.
\end{split}
\end{equation*}
Using the relation $|1-e^{-x}|\lesssim |x|$, $|\sin{\left((k_i-k_j')\cdot x\right)}|\leq 1$, and $|(v_{i}\cdot k'_{j})|\approx |k_j'|/|k_i'|$ we have
\begin{equation*}
\begin{split}
|\rho_{1,1}| \lesssim& r^{-2\beta} \sum_{i=1}^r |k_i|^{2}e^{-t|k_i-k_j'|^2}t + r^{-2\beta}\sum_{j=1}^r \sum_{i<j}  |k_j'|^{2}e^{-t|k_i-k_j'|^2}t.
\end{split}
\end{equation*}
Noting that $|k_i|\approx|k_i'|$, $|k_i-k_j'|\approx |k_i|$ for $j<i$, $|k_i-k_j'|\approx |k_j'|$ for $j>i$, and $\sum_{j<i}|k_j|^s \approx |k_i|^s$ for all $s>0$, we get
\begin{equation*}
|\rho_{1,1}| \lesssim r^{-2\beta} \sum_{i=1}^r|k_i|^2 e^{-t|k_i|^2}t + r^{-2\beta}\sum_{j=1}^rj|k_j'|^2 e^{-t|k_j'|^2}t \lesssim r^{-2\beta} t^{-\delta}.
\end{equation*}
Here the last inequality follows similar steps as in \eqref{est-1206-2244}. The lemma is proved.
\end{proof}
From Lemma \ref{lem-1202-2023-1} and \ref{lem-1202-2023-2}, we deduce
\begin{equation}\label{est-1206-2309-1}
|\rho_1| \leq |\rho_{1,0}| + |\rho_{1,1}| + |\rho_{1,2}|\lesssim r^{1-2\beta} 
	+ r^{-2\beta}t^{-\delta} ,\,\,\text{for}~t\in (0,1].
\end{equation}
From this, we obtain
\begin{equation}\label{est-1204-1621}
|B_2(e^{t\Delta}u_{0}, e^{t\Delta}\rho_{0})| \lesssim t|B_3(e^{t\Delta}u_{0}, e^{t\Delta}\rho_{0})| = t |\rho_1| \lesssim  r^{1-2\beta} t
+  r^{-2\beta}t^{1-\delta},\,\,\text{for}~t\in (0,1].
\end{equation}
One can simply check this by comparing $B_2$ and $B_3$ and noting that $|\mathbb{P}\sin(k\cdot x)\vec{e}_i|\lesssim |\sin(k\cdot x)|$.

Combining \eqref{est-1204-2222} and \eqref{est-1204-1621}, we obtain that for~$t\in (0,1]$,
\begin{align}\label{est-1206-2309-2}
\begin{split}
\|u_1(\cdot,t)\|_{L^\infty}
\lesssim |B_1(e^{t\Delta}u_{0}, e^{t\Delta}u_{0})| + |B_2(e^{t\Delta}u_{0}, e^{t\Delta}\rho_{0})|
&\lesssim 
r^{-2\beta} t^{-\delta} + r^{1-2\beta}t 
	+ r^{-2\beta}t^{1-\delta}
\\&\lesssim 
r^{-2\beta} t^{-\delta} + r^{1-2\beta}t^{1-\delta}.
\end{split}
\end{align}

\subsection{Analysis of $y$ and $z$}\label{sec-1207-1609-2}
In this section, we estimate $y$ and $z$. Throughout this section, we denote
\begin{equation*}
M(t):= \sup_{s\in(0,t)}\{s^\delta\|y(\cdot,s)\|_{L^\infty}\} + t\sup_{s\in(0,t)}\{s^\delta\|z(\cdot,s)\|_{L^\infty}\}.
\end{equation*}
The main result of this section is the following estimate.
\begin{Proposition}\label{prop-1204-1823}
For some $\nu>0$, let $T:=r^{-\nu}$. Then for sufficiently large $r$,
\begin{equation}\label{eqn-1203-1900}
M(t) \lesssim r^{-3\beta} + r^{1-3\beta}t^{1+\delta} + r^{2-4\beta}t^{5/2+\delta}, \quad \text{for}~t\in(0,T],
\end{equation}
provided that
\begin{equation}\label{eqn-1203-2005}
\beta>\max\{0, \frac{1}{2}-\frac{3}{4}\nu\}.
\end{equation}
In particular,
\begin{equation*}
\|z(\cdot,t)\|_{L^\infty} \lesssim r^{-3\beta}t^{-1-\delta} + r^{1-3\beta} + r^{2-4\beta}t^{3/2}
\end{equation*}
for any $t\in(0,T)$.
\end{Proposition}
\begin{proof}
We first estimate $\|y(\cdot,t)\|_{L^\infty}$ and $\|z(\cdot,t)\|_{L^\infty}$. For this, we apply the bilinear estimates in Lemma~\ref{lem-1204-1626}, \eqref{eqn-1208-2243}, and \eqref{eqn-1209-0022} to each term in \eqref{eqn-1204-1626}, with the help of \eqref{est-1204-1546}, \eqref{est-1206-2309-1}, and \eqref{est-1206-2309-2}.

\textbf{Estimating $y$}. Recall $y=y_0+y_1+y_2$ and their expressions \eqref{eqn-1204-1626}. We first estimate $y_0$ which does not contain $y$ or $z$, term by term. Noting \eqref{est-1204-1546} and \eqref{est-1206-2309-2}, we have
	     \begin{equation}	     \label{eqn-1204-1235}
	     \begin{split}
	     \Vert  B_1 (g, u_1)\Vert_{L^{\infty}}
	    & \les
	     \int_{0}^{t} \frac{1}{(t-s)^{1/2}} 
	     \Vert  e^{t\Delta}u_{0}\Vert_{L^{\infty}}
	     \Vert  u_1\Vert_{L^{\infty}}
	     \,ds
	     \les
	     \int_{0}^{t} \frac{1}{(t-s)^{1/2}} 
	     r^{-\beta}s^{-1/2}
	     (r^{-2\beta}s^{-\delta} + r^{1-2\beta}s)
	     \,ds
	     \\&\les
	     r^{-3\beta}\int_{0}^{t} \frac{1}{(t-s)^{1/2}} 
	     s^{-1/2-\delta}
	     \,ds + r^{1-3\beta}\int_{0}^{t} \frac{1}{(t-s)^{1/2}} s^{1/2}\,ds
	     \\&\les
	     r^{-3\beta}t^{-\delta} + r^{1-3\beta}t
	     .
	     \end{split}
	     \end{equation}
Using \eqref{est-1206-2309-2}, we obtain
	     \begin{equation}	     \label{eqn-1204-1236}
	     \begin{split}
	     \Vert  B_1 (u_1, u_1)\Vert_{L^{\infty}}
	     & \les
	     \int_{0}^{t} \frac{1}{(t-s)^{1/2}} 
	     \Vert  u_1\Vert_{L^{\infty}}
	     \Vert  u_1\Vert_{L^{\infty}}
	     \,ds \les \int_{0}^{t} \frac{1}{(t-s)^{1/2}} \big((r^{-2\beta}s^{-\delta})^2 + (r^{1-2\beta}s)^2\big)\,ds\\
	     &\les
	     r^{-4\beta} t^{1/2 - 2\delta} + r^{2-4\beta}t^{5/2}
	     .
	     \end{split}
	     \end{equation}
	     Similarly, using \eqref{eqn-1208-2243}, \eqref{est-1204-1546}, \eqref{est-1206-2309-1}, and \eqref{est-1206-2309-2}, we obtain the $B_2$ terms estimates
	     \begin{equation}	     \label{eqn-1204-1237}
	     \begin{split}
	     \Vert  B_2 (g, \rho_1)\Vert_{L^{\infty}}
	     & \les
	     \int_{0}^{t} (t-s)^{1/2}
	     \Vert  e^{t\Delta}u_{0}\Vert_{L^{\infty}}
	     \Vert  \rho_1\Vert_{L^{\infty}}
	     \,ds \les  \int_{0}^{t} (t-s)^{1/2} r^{-\beta}s^{-1/2} (r^{-2\beta}s^{-\delta} + r^{1-2\beta})\,ds\\
&\les
r^{-3\beta}t^{1-\delta} + r^{1-3\beta}t,
	     \end{split}
	     \end{equation}
	     \begin{equation}	     \label{eqn-1204-1238}
	     \begin{split}
	     \Vert  B_2 (u_1, \theta)\Vert_{L^{\infty}}
	     & \les
	     \int_{0}^{t} (t-s)^{1/2}
	     \Vert  u_{1}\Vert_{L^{\infty}}
	     \Vert  e^{t\Delta}\rho_{0}\Vert_{L^{\infty}}
	     \,ds
	     \les
	     \int_{0}^{t} (t-s)^{1/2}
	     (r^{-2\beta}s^{-\delta} + r^{1-2\beta}s)
         r^{-\beta}s^{-1/2}
	     \,ds
	     \\&\les
	     r^{-3\beta}t^{1-\delta}+ r^{1-3\beta}t^{2}
	     ,
	     \end{split}
	     \end{equation}
and
	     \begin{equation}	     \label{eqn-1204-1239}
	     \begin{split}
	     \Vert  B_2 (u_1, \rho_1)\Vert_{L^{\infty}}
	     & \les
	     \int_{0}^{t} (t-s)^{1/2}
	     \Vert  u_{1}\Vert_{L^{\infty}}
	     \Vert  \rho_1\Vert_{L^{\infty}} \les 	\int_{0}^{t} (t-s)^{1/2} (r^{-2\beta}s^{-\delta} + r^{1-2\beta}s)(r^{1-2\beta} + r^{-2\beta}s^{-\delta})\,ds\\
	     &\les
	     r^{1-4\beta} t^{3/2 - \delta} 
	     	+
	     	r^{2-4\beta}t^{5/2}
	     	+
	     	r^{-4\beta}t^{3/2 - 2\delta}
           .
	     \end{split}
	     \end{equation}
Again, in the last inequality we use $t\in(0,1]$. Therefore, combining the above estimates \eqref{eqn-1204-1235}-- \eqref{eqn-1204-1239} yields
	     \begin{align}
	     \Vert  y_{0}\Vert_{L^{\infty}}
	      \les 
	      r^{-3\beta}t^{-\delta} + r^{1-3\beta}t
     	+
     	r^{2-4\beta}t^{5/2},\quad t\in (0,1].
	     \label{est-1204-1656}
	     \end{align}
	     Next, we estimate $y_{1}$ which contains linear combinations of $y$ and $z$. This is
	     \begin{align}
	     \begin{split}
	     \Vert  y_{1}\Vert_{L^{\infty}}
	     & \leq
	     \Vert  B_1(g,y)\Vert_{L^{\infty}}
	     +
	     \Vert  B_1(u_1,y)\Vert_{L^{\infty}}
	     +
	     \Vert  B_2(g,z)\Vert_{L^{\infty}}
	     +
	     \Vert  B_2(y,\theta)\Vert_{L^{\infty}}
	     \\&\indeq +
	     \Vert  B_2(y,\rho_1)\Vert_{L^{\infty}}
	     + 
	     \Vert  B_2(u_1,z)\Vert_{L^{\infty}}
	     .
	     \label{eqn-1204-1240}
	     \end{split}
	     \end{align}
	     Similarly to the estimates of \eqref{eqn-1204-1235} and \eqref{eqn-1204-1236}, we get
	     \begin{align}
	     \begin{split}
	     \Vert  B_1 (g, y)\Vert_{L^{\infty}}
	     & \les
	     \int_{0}^{t} \frac{1}{(t-s)^{1/2}} 
	     \Vert  e^{t\Delta}u_{0}\Vert_{L^{\infty}}
	     \Vert  y\Vert_{L^{\infty}}
	     \,ds
	     \les
	     r^{-\beta}\int_{0}^{t} \frac{1}{(t-s)^{1/2}} 
	     s^{-1/2-\delta}\,ds
	     \sup_{0<s<t}\{s^{\delta}\Vert  y\Vert_{L^{\infty}}\}
	     \\&\les
	     r^{-\beta}t^{-\delta}
	     \sup_{0<s<t}\{s^{\delta}\Vert  y\Vert_{L^{\infty}}\}
	     ,
	     \label{eqn-1204-1241}
	     \end{split}
	     \end{align}
and
	     \begin{align}
	     \begin{split}
	     \Vert  B_1 (u_1, y)\Vert_{L^{\infty}}
	     & \les
	     \int_{0}^{t} \frac{1}{(t-s)^{1/2}} 
	     \Vert  u_{1}\Vert_{L^{\infty}}
	     \Vert  y\Vert_{L^{\infty}}
	     \,ds
	     \les
	     \int_{0}^{t} \frac{1}{(t-s)^{1/2}} 
	      (r^{-2\beta}s^{-\delta} + r^{1-2\beta}s)s^{-\delta}\,ds
	     \sup_{0<s<t}\{s^\delta \Vert  y\Vert_{L^{\infty}}\}
	     \\&\les
	    (r^{-2\beta} t^{1/2-2\delta} + r^{1-2\beta}t^{3/2-\delta})
	     \sup_{0<s<t}\{s^\delta \Vert  y\Vert_{L^{\infty}}\}
	     .
	     \label{eqn-1204-1242}
	     \end{split}
	     \end{align}
Similarly to \eqref{eqn-1204-1237}-\eqref{eqn-1204-1239}, we obtain
	     \begin{align}
	     \begin{split}
	     \Vert  B_2 (g, z)\Vert_{L^{\infty}}
	     & \les
	     \int_{0}^{t} (t-s)^{1/2}
	     \Vert  e^{t\Delta}u_{0}\Vert_{L^{\infty}}
	     \Vert  z\Vert_{L^{\infty}}
	     \,ds
	     \les
	     r^{-\beta}\int_{0}^{t} (t-s)^{1/2}s^{-1/2-\delta}
	     \,ds
	     \sup_{0<s<t}\{s^\delta\Vert  z\Vert_{L^{\infty}}\}
	     \\&\les
	     r^{-\beta}t^{1-\delta}
	     \sup_{0<s<t}\{s^\delta\Vert z\Vert_{L^{\infty}}\}
	     ,
	     \label{eqn-1204-1243}
	     \end{split}
	     \end{align}
	     \begin{align}
	     \begin{split}
	     \Vert  B_2 (y, \theta)\Vert_{L^{\infty}}
	     & \les
	     \int_{0}^{t} (t-s)^{1/2}
	     \Vert  y\Vert_{L^{\infty}}\Vert  e^{t\Delta}\rho_{0}\Vert_{L^{\infty}}
	     \,ds
	     \les
	     r^{-\beta}\int_{0}^{t} (t-s)^{1/2}
	     s^{-1/2-\delta}\,ds
	     \sup_{0<s<t}\{s^\delta\Vert  y\Vert_{L^{\infty}}\}
	     \\&\les
	     r^{-\beta}t^{1-\delta}
	     \sup_{0<s<t}\{s^\delta\Vert  y\Vert_{L^{\infty}}\}
	     ,
	     \label{eqn-1204-1244}
	     \end{split}
	     \end{align}
	     \begin{align}
	     \begin{split}
	     \Vert  B_2 (y, \rho_1)\Vert_{L^{\infty}}
	     & \les
	     \int_{0}^{t} (t-s)^{1/2}
	     \Vert  \rho_1\Vert_{L^{\infty}}
	     \Vert  y\Vert_{L^{\infty}}
	     \,ds \les
\int_{0}^{t} (t-s)^{1/2}(r^{-2\beta}s^{-2\delta} + r^{1-2\beta} s^{-\delta})
	     \,ds
	     \sup_{0<s<t}\{s^\delta\Vert  y\Vert_{L^{\infty}}\}
	     \\&\les
	   (r^{-2\beta}t^{3/2-2\delta} + r^{1-2\beta}t^{3/2-\delta} )
	     \sup_{0<s<t}\{s^{\delta}\Vert  y\Vert_{L^{\infty}}\}
	     ,
	     \label{eqn-1204-1245}
	     \end{split}
	     \end{align}
and
	     \begin{align}
	     \begin{split}
	     \Vert  B_2 (u_1, z)\Vert_{L^{\infty}}
	     & \les
	     \int_{0}^{t} (t-s)^{1/2}
	     \Vert  u_{1}\Vert_{L^{\infty}}
	     \Vert  z\Vert_{L^{\infty}}
	     \,ds \les \int_{0}^{t} (t-s)^{1/2} (r^{-2\beta}s^{-2\delta} + r^{1-2\beta}s^{1-\delta})\,ds\sup_{0<s<t}\{s^\delta\Vert  z\Vert_{L^{\infty}}\}\\
	     &\les
	     (r^{-2\beta} t^{3/2-2\delta} + r^{1-2\beta}t^{5/2-\delta})\sup_{0<s<t}\{s^\delta\Vert  z\Vert_{L^{\infty}}\}
	     .
	     \label{eqn-1204-1246}
	     \end{split}
	     \end{align}
	     Therefore, combining the above estimates \eqref{eqn-1204-1240}-- \eqref{eqn-1204-1246} yields
\begin{equation}\label{est-1204-1701}
	     \Vert  y_{1}\Vert_{L^{\infty}} \les (r^{-\beta}t^{-\delta} + r^{1-2\beta}t^{3/2-\delta})
\sup_{0<s<t}\{s^{\delta}\Vert  y\Vert_{L^{\infty}}\} + (r^{-\beta}t^{-\delta} + r^{1-2\beta}t^{3/2-\delta})t\sup_{0<s<t}\{s^{\delta}\Vert  z\Vert_{L^{\infty}}\}.
\end{equation}
%
	     Next, we treat $y_2$, i.e., the quadratic combinations in terms of $y$ or $z$,
	     \begin{align}
	     \begin{split}
	     \Vert  y_{2}\Vert_{L^{\infty}}
	     &\les
	     \Vert  B_1(y,y)\Vert_{L^{\infty}}
	     +
	     \Vert  B_2(y,z)\Vert_{L^{\infty}}\\
	     &\les
	     \int_{0}^{t} \frac{1}{(t-s)^{1/2}} 
	     \Vert  y\Vert_{L^{\infty}}^2
	     \,ds
	     +
	      \int_{0}^{t} (t-s)^{1/2}
	     \Vert  y\Vert_{L^{\infty}}
	     \Vert  z\Vert_{L^{\infty}}
	     \,ds
	     \\&\les
	     t^{1/2-2\delta}(\sup_{0<s<t}\{s^\delta\Vert  y\Vert_{L^{\infty}}\})^{2}
	     +
	     t^{3/2-2\delta}\sup_{0<s<t}\{s^\delta\Vert  y\Vert_{L^{\infty}}\}
	     \sup_{0<s<t}\{s^\delta\Vert z\Vert_{L^{\infty}}\}
	     .
	     \label{est-1204-1659}
	     \end{split}
	     \end{align}
Combining \eqref{est-1204-1656},\eqref{est-1204-1701}, and \eqref{est-1204-1659}, for $t\in(0,1]$ we obtain
\begin{equation*}
\|y\|_{L^\infty}\lesssim r^{-3\beta}t^{-\delta} + r^{1-3\beta}t + r^{2-4\beta}t^{5/2} + (r^{-\beta}t^{-\delta} + r^{1-2\beta}t^{3/2-\delta})M(t) + t^{1/2-2\delta}M(t)^2.
\end{equation*}
Hence, for any $t\in(0,1]$,
\begin{equation}\label{est-1204-1357}
\sup_{0<s<t}\{s^{\delta}\|y\|_{L^\infty}\}\lesssim r^{-3\beta} + r^{1-3\beta}t^{1+\delta} + r^{2-4\beta}t^{5/2+\delta} + (r^{-\beta} + r^{1-2\beta}t^{3/2})M(t) + t^{1/2-\delta}M(t)^2.
\end{equation}
\textbf{Estimating $z$}. Again, recall that $z=z_0+z_1+z_2$ along with their expressions in \eqref{eqn-1204-1626}. For $z_0$, we have
	     \begin{align}
	     \begin{split}
	     \Vert  z_{0}\Vert_{L^{\infty}}
	     \leq
	     \Vert  B_3(g,\rho_1)\Vert_{L^{\infty}}
	     +
	     \Vert  B_3(u_1,\theta)\Vert_{L^{\infty}}
	     +
	     \Vert  B_3(u_1,\rho_1)\Vert_{L^{\infty}}
	     .
	     \llabel{EQ19}
	     \end{split}
	     \end{align}
Noting $|B_3(f,g)|\les t^{-1}|B_2(f,g)|$ and the estimates \eqref{eqn-1204-1237}, \eqref{eqn-1204-1238}, and \eqref{eqn-1204-1239}, we obtain
	     \begin{equation*}	    
	     \Vert  B_3 (g, \rho_1)\Vert_{L^{\infty}} \les t^{-1} \Vert  B_2 (g, \rho_1)\Vert_{L^{\infty}} \les
         r^{-2\beta-\beta t^{-\delta}} + r^{1-3\beta},
	     \end{equation*}
	     \begin{equation*}
	     \Vert  B_3 (u_{1}, \theta)\Vert_{L^{\infty}} \les t^{-1}\Vert  B_2 (u_{1}, \theta)\Vert_{L^{\infty}} \les r^{-3\beta}t^{-\delta}+ r^{1-3\beta}t,
	     \end{equation*}
and
	     \begin{equation*}
	     \Vert  B_3 (u_{1}, \rho_1)\Vert_{L^{\infty}} \les t^{-1}\Vert  B_2 (u_{1}, \rho_1)\Vert_{L^{\infty}} \les r^{1-4\beta} t^{1/2 - \delta} + r^{2-4\beta}t^{3/2} + r^{-4\beta}t^{1/2 - 2\delta}.
	     \end{equation*}
	     Therefore, 
	     \begin{equation}\label{est-1203-1605}
	     \Vert z_{0}\Vert_{L^{\infty}} \lesssim r^{-3\beta}t^{-\delta} + r^{1-3\beta} + r^{2-4\beta}t^{3/2}.
	     \end{equation}
	     Next we treat the linear term $z_1$:
	     \begin{equation*}
	     \Vert  z_{1}\Vert_{L^{\infty}}
	     \leq
	     \Vert  B_3(g, z)\Vert_{L^{\infty}}
	     +
	     \Vert  B_3(y,\theta)\Vert_{L^{\infty}}
	     +
	     \Vert  B_3(y,\rho_1)\Vert_{L^{\infty}}
	     +
	     \Vert  B_3(u_{1}, z)\Vert_{L^{\infty}}
	     .
	     \end{equation*}
Similarly to the estimates for $z_0$, noting \eqref{eqn-1204-1243}, \eqref{eqn-1204-1244}, \eqref{eqn-1204-1245}, and \eqref{eqn-1204-1246}, we obtain
	     \begin{equation*}
	     \Vert  B_3 (g, z)\Vert_{L^{\infty}} \les t^{-1}B_2 (g, z) \les r^{-\beta}t^{-\delta}
	     \sup_{0<s<t} \{s^\delta\Vert z\Vert_{L^{\infty}}\},
	     \end{equation*}
	     \begin{equation*}
	     \Vert  B_3 (y, \theta)\Vert_{L^{\infty}}\les t^{-1}\Vert  B_3 (y, \theta)\Vert_{L^{\infty}} \les r^{-\beta}t^{-\delta}
	     \sup_{0<s<t} \{s^\delta\Vert  y\Vert_{L^{\infty}}\},
	     \end{equation*}
	     \begin{equation*}
	     \Vert  B_3 (y, \rho_1)\Vert_{L^{\infty}} \les t^{-1}\Vert  B_2 (y, \rho_1)\Vert_{L^{\infty}} \les
	    (r^{-2\beta}t^{1/2 - 2\delta} + r^{1-2\beta}
	     t^{1/2-\delta}) \sup_{0<s<t} \{s^\delta\Vert  y\Vert_{L^{\infty}}\},
	     \end{equation*}
and
	     \begin{equation*}
	     \Vert  B_3 (u_{1}, z)\Vert_{L^{\infty}} \les t^{-1}\Vert  B_2 (u_{1}, z)\Vert_{L^{\infty}} \les (r^{-2\beta} t^{1/2-2\delta} + r^{1-2\beta}t^{3/2-\delta}) \sup_{0<s<t}\{s^\delta\Vert  z\Vert_{L^{\infty}}\}.
	     \end{equation*}
	     Therefore, by combining the above estimates, we have
	     \begin{equation}\label{est-1204-1348}
	     \Vert  z_{1}\Vert_{L^{\infty}} \les (r^{-\beta}t^{-\delta} + r^{1-2\beta}t^{1/2-\delta})
\sup_{0<s<t}\{s^{\delta}\Vert  y\Vert_{L^{\infty}}\} + (r^{-\beta}t^{-\delta} + r^{1-2\beta}t^{3/2-\delta})\sup_{0<s<t}\{s^{\delta}\Vert  z\Vert_{L^{\infty}}\}.
	     \end{equation}
	     Next we treat the quadratic term $z_2$ as
	     \begin{equation}\label{est-1203-1606-2}
	     \Vert  z_{2}\Vert_{L^{\infty}}
	     =
	     \Vert  B_3 (y, z)\Vert_{L^{\infty}}
	     \les
	      \int_{0}^{t} \frac{1}{(t-s)^{1/2}} 
	     \Vert  y\Vert_{L^{\infty}}
	     \Vert  z\Vert_{L^{\infty}}
	     \,ds
	     \les
	     t^{1/2-2\delta}
	    \sup_{0<s<t} \{s^\delta\Vert  y\Vert_{L^{\infty}}\}
	    \sup_{0<s<t} \{\Vert s^\delta z\Vert_{L^{\infty}}\}
	     .
	     \end{equation}
Summing \eqref{est-1203-1605}, \eqref{est-1204-1348}, and \eqref{est-1203-1606-2}, multiplying $t^\delta$, and taking $\sup$ in $t$, we obtain
\begin{equation}\label{est-1204-1356}
\begin{split}
\sup_{0<s<t}\{s^\delta\|z\|_{L^\infty}\}&\lesssim 	     r^{-3\beta} + r^{1-3\beta}t^\delta + r^{2-4\beta}t^{3/2+\delta}\\
&\indeq + (r^{-\beta} + r^{1-2\beta}t^{1/2})
\sup_{0<s<t}\{s^{\delta}\Vert  y\Vert_{L^{\infty}}\} + (r^{-\beta} + r^{1-2\beta}t^{3/2})\sup_{0<s<t}\{s^{\delta}\Vert  z\Vert_{L^{\infty}}\}\\
&\indeq + t^{1/2-\delta}
	    \sup_{0<s<t} \{s^\delta\Vert  y\Vert_{L^{\infty}}\}
	    \sup_{0<s<t} \{\Vert s^\delta z\Vert_{L^{\infty}}\}.
\end{split}
\end{equation}
Now, multiplying \eqref{est-1204-1356} by $t$, then adding to \eqref{est-1204-1357}, for any $t\in(0,1]$ we have
\begin{equation*}
M(t)\lesssim r^{-3\beta} + r^{1-3\beta}t^{1+\delta} + r^{2-4\beta}t^{5/2+\delta} + (r^{-\beta} + r^{1-2\beta}t^{3/2})M(t) + t^{1/2-\delta}M(t)^2.
\end{equation*}
The following lemma should be read as part of the proof.
\begin{Lemma}\label{lem-1203-1819}
Suppose for $0<t<T$, the non-negative continuous function $M(t)$ satisfies $M(0)=0$ and
\begin{equation*}
M(t)\lesssim C_1(t) + C_2(t) M(t) + C_3(t)M(t)^2,
\end{equation*}
where $C_{1}(t),C_{2}(t)$, and $C_{3}(t)$ are all non-negative. Then for any $t\in(0,T)$
\begin{equation*}
M(t) \lesssim C_1(t),
\end{equation*}
provided that $C_{1}(t),C_{2}(t)$, and $C_{3}(t)$ satisfy
\begin{equation}\label{eqn-1203-1821}
\sup_{t<T}\{C_{2}(t) + C_{1}(t)C_{3}(t)\} \ll 1.
\end{equation}
\end{Lemma}
The proof is by a standard absorbing argument which we will omit here. Next we return to the proof of the proposition. The estimate \eqref{eqn-1203-1900} can be obtained by applying Lemma \ref{lem-1203-1819} with
\begin{equation*}
C_1 = r^{-3\beta} + r^{1-3\beta}t^{1+\delta} + r^{2-4\beta}t^{5/2+\delta},\quad C_2 = r^{-\beta} + r^{1-2\beta}t^{3/2}, \quad\text{and}\,\, C_3=t^{1/2-\delta}.
\end{equation*}
Hence, we only need to check the condition \eqref{eqn-1203-1821}. We see that $C_1, C_2$, and $C_3$ are all increasing in $t$. Thus it suffices to check
\begin{equation*}
C_2(T) + C_1(T)C_3(T)\ll 1.
\end{equation*}
Substituting $T=r^{-\nu}$ in, we can rewrite it as
\begin{equation*}
r^{-\beta} + r^{1-2\beta-\frac{3}{2}\nu} + r^{-3\beta -(-\frac{1}{2}-\delta)\nu} + r^{1-3\beta-\frac{3}{2}\nu} + r^{2-4\beta-3\nu}\ll 1.
\end{equation*}
One can check that all the exponents are negative, given \eqref{eqn-1203-2005}. Hence the condition \eqref{eqn-1203-1821} holds if we choose $r$ sufficiently large. 
\end{proof}


\section{Closing the estimates}\label{sec-prof-main}
In this section, we close the estimates and prove our main theorem. 
\begin{proof}[Proof of Theorem \ref{thm-1204-1740}]
First, from \eqref{est-1203-1959}, we can choose $r$ sufficiently large to achieve
\begin{equation*}
\|u_0\|_{\dot{B}^{-1}_{\infty, \infty}}\leq \epsilon,\quad  \|\rho_0\|_{\dot{B}^{-1}_{\infty, \infty}}\leq \epsilon.
\end{equation*}
It remains to prove \eqref{est-1204-1822}. If $\beta>\max\{0, \frac{1}{2}-\frac{3}{4}\nu\}$ and $r\gg 1$, we have the estimate in Proposition \ref{prop-1204-1823}. As in Proposition \ref{prop-1204-1823}, we choose $T=r^{-\nu}$. Furthermore, in \eqref{eqn-1204-1824} we choose $K=r^{\nu/2}$ so that Lemma \ref{lem-1202-2023-1} holds for $t=r^{-\nu}$. Now, applying Lemma \ref{lem-1202-2023-1}, Lemma \ref{lem-1204-1830}, Lemma \ref{lem-1202-2023-2}, and Proposition \ref{prop-1204-1823} at the time $t=r^{-\nu}$, we obtain
	     \begin{equation}
	     \begin{split}
	     \Vert  \rho \Vert_{\dot{B}^{-s}_{\infty, \infty}}
	     &\geq
	     \Vert  \rho_{10} \Vert_{\dot{B}^{-s}_{\infty, \infty}}
	     -
	     \Vert  e^{t\Delta}\rho_{0} \Vert_{L^{\infty}}
	     -
	     \Vert  \rho_{11} \Vert_{L^{\infty}}
	     -
	     \Vert  \rho_{12} \Vert_{L^{\infty}}
	     -
	     \Vert  z \Vert_{L^{\infty}}
	     \nonumber\\
&\gtrsim
	     r^{1-2\beta} 
	     -
	     r^{-\beta}t^{-1/2}
	     -
	     r^{-2\beta}t^{-\delta}
	     - r^{-3\beta}t^{-1-\delta} - r^{1-3\beta} - r^{2-4\beta}t^{3/2}\nonumber\\
&\gtrsim
	     r^{1-2\beta} (1-r^{-1+\beta}t^{-1/2} - r^{-1}t^{-\delta} - r^{-1-\beta}t^{-1-\delta} - r^{-\beta} - r^{1-2\beta}t^{3/2})\\
&\gtrsim r^{1-2\beta} (1-r^{-1+\beta-\nu/2} - r^{-1+\nu\delta} - r^{-1-\beta+\nu+\nu\delta} - r^{-\beta} - r^{1-2\beta-(3/2)\nu} ).\label{est-1203-1929}
	     \end{split}
	     \end{equation}
In order to make 
\begin{equation*}
1-r^{-1+\beta-\nu/2} - r^{-1+\nu\delta} - r^{-1-\beta+\nu+\nu\delta} - r^{-\beta} - r^{1-2\beta-(3/2)\nu} \approx 1\quad \text{and}\,\,r^{1-2\beta}\gg 1
\end{equation*}
$r\gg 1$, we need to choose parameters satisfying
\begin{equation}\label{eqn-1204-1901}
\begin{split}
&-1+\beta-\nu/2<0,\,\, -1+\nu\delta <0,\,\, -1-\beta+\nu+\nu\delta<0,\,\, 1-2\beta-(3/2)\nu<0,\\
&\text{and}\,\,1-2\beta>0.
\end{split}
\end{equation}
Once this holds, we may choose sufficiently large $r$ so that $\Vert\rho \Vert_{\dot{B}^{-s}_{\infty, \infty}}>1/\epsilon$. Indeed, there exist parameters satisfying \eqref{eqn-1203-2005} and \eqref{eqn-1204-1901}. Recall that $\delta\ll 1$ is fixed. Here we can choose
\begin{equation*}
\beta=1/2-\nu/2,\quad\text{and}\,\, \nu\ll 1.
\end{equation*}
In particular, for such choice of parameters,
$$T=r^{-\nu}\approx \Vert  \rho \Vert_{\dot{B}^{-s}_{\infty, \infty}}^{-1}<\epsilon,$$
which completes the proof.
\end{proof}

	\section{Acknowledgments} 
	The authors would like to thank Hongjie Dong and Igor Kukavica for useful comments and suggestions. Z. Li was partially supported by the NSF under agreement DMS-1600593. W. Wang was supported in part by the NSF grant DMS-1907992.

\end{document}